\documentclass{article}

\usepackage{amsmath}
\usepackage{amssymb}
\usepackage{amsthm}

\usepackage{times}

\newtheorem{lemma}{Lemma}[section]
\newtheorem{theorem}[lemma]{Theorem}
\newtheorem{definition}[lemma]{Definition}

\newcommand{\gin}{\tilde\in}
\newcommand{\RCA}{$\mathbf{RCA_0}$}
\newcommand{\ACA}{$\mathbf{ACA_0}$}
\newcommand{\ATR}{$\mathbf{ATR_0}$}

\newcommand{\Sigmati}{$\mathbf{\Sigma^1_1-TI_0}$}

\author{Henry Towsner}
\title{Hindman's Theorem: An Ultrafilter Argument in Second Order Arithmetic}
\date{\today}

\begin{document}
\maketitle

\begin{abstract}
  Hindman's Theorem is a prototypical example of a combinatorial theorem with a proof that uses the topology of the ultrafilters.  We show how the methods of this proof, including topological arguments about ultrafilters, can be translated into second order arithmetic.
\end{abstract}

\section{Introduction}
The topology of the ultrafilters on the natural numbers---the Stone-\v{C}ech compactification---has long been a powerful tool in combinatorics (\cite{hindman98} gives an extensive treatment).  Such proofs, however, are invariably quite infinitary: the construction of even a single nonprincipal ultrafilter requires (a weak form of) the axiom of choice, and typical proofs involve not just arbitrary ultrafilters, but special classes of them which require further applications of the axiom of choice to obtain.

The first application of this method was to Hindman's Theorem, which states that in any finite partition of the natural numbers, some element of the partition contains infinitely many integers and their finite sums.  There are three standard proofs: Hindman's original combinatorial argument \cite{hindman74}, Baumgartner's streamlined combinatorial argument \cite{baumgartner74}, and the Galvin-Glazer proof using ultrafilters (see \cite{comfort77} or \cite{hindman98}).  The first two are thoroughly analyzed in \cite{blass87}, where it is shown that the Baumgartner proof may be formalized in the formal system $\mathbf{\Pi^1_2-TI_0}$ and the Hindman proof in the weaker system $\mathbf{ACA_0^+}$.  The strongest reversal obtained is that Hindman's Theorem implies \ACA{} over \RCA, leaving a gap in the strength of the theorem which remains unresolved.

Glazer's proof, however, is left untouched, since it appears to require the use of fourth order objects (closed semigroups of ultrafilters).  Hirst \cite{hirst04} has used a natural coding of these closed semigroups by sets of integers to show that the existence of a weakened notion of ultrafilter (a ``partial'' ultrafilter---a filter which is guaranteed to contain either $A$ or $A^c$ whenever $A$ comes from a countable collection fixed in advance) is sufficient to imply not only Hindman's Theorem, but also an iterated form of Hindman's Theorem.

Here we answer a question asked by Hirst \cite{birs08}, among others, by showing how to carry out an analog of Glazer's proof in second order arithmetic.  While we note the reverse mathematical strength of our main arguments, we do not discuss the reverse mathematical concerns further; thorough explanations of all reverse mathematics referenced can be found in \cite{simpson99}.

We further explore the consequences of this method in \cite{towsner09:hindman_simple}, where we give a short, explicit proof of Hindman's Theorem within $\mathbf{ACA_0^+}$, similar to Hindman's original proof but based on the methods used here.  

We are grateful to Mathias Beiglb\"ock for many helpful discussions about the many facets of Hindman's Theorem.

\section{General Definitions}
Our basic definitions are modeled on those in \cite{hirst04}.  Since other sections of the proof already force us to work in a system much stronger that \ACA, we sometimes take definitions which are equivalent to Hirst's in \ACA{} but not in \RCA.

\begin{definition}
\begin{itemize}
\item   Given a set $X$, $X-n:=\{m\mid m+n\in X\}$.
\item   If $U=\{ U_i\mid i\in\mathbb{N}\}$ is a sequence of subsets of $\mathbb{N}$ and $F\subseteq\mathbb{N}$ is finite, we write $U_F:=\bigcap_{i\in F}U_i$.
\item We write $X\gin U$ if there is a finite $F\subseteq\mathbb{N}$ such that $U_F\subseteq X$.
\item  A countable sequence $U=\{ U_i\mid i\in\mathbb{N}\}$ of subsets of $\mathbb{N}$ \emph{satisfies the finite intersection property} (``satisfies fip'') if for any finite $F\subseteq\mathbb{N}$, $U_F$ is infinite.
\item If $U,V$ satisfy fip, we define $X\gin U+V$ if there is a $Y\gin V$ such that for every $n\in Y$, $X-n\gin U$.
\item We say $U$ is a \emph{semigroup} if $U$ satisfies fip and $X\in U$ implies $X\gin U+U$.
\end{itemize}
\end{definition}
A sequence $U$ satisfying fip can be seen as a code for a non-empty closed set in the topology of $\beta\mathbb{N}$, the Stone-\v{C}ech compactification of the natural numbers.  (That is, the space of ultrafilters over $\mathbb{N}$.)  Namely, $U$ should be viewed as the set of ultrafilters $p$ such that $U\subseteq p$.  Saying that $U$ is a semigroup in our sense is precisely saying that the closed set coded by $U$ is in fact a semigroup.  We can think of closed sets as ``approximate ultrafilters'', deciding only some sets while leaving others ambiguous, and closed semigroups as ``approximate idempotents''.

\begin{definition}
  Let $S$ be a set of integers.  The finite sums from $S$, $FS(S)$, are defined by
\[FS(S):=\{\sum_{i\in F}i\mid F\subseteq S\wedge F\text{ is finite}\}.\]
It is also convenient to define $NS(S):=FS(S)\setminus\{0\}$, the non-empty finite sums from $S$.
\end{definition}

Our goal is to prove:
\begin{theorem}[Hindman's Theorem]
  For every finite partition of the natural numbers, $\mathbb{N}=C_1\cup\cdots\cup C_n$, there is some $i$ and an infinite set $X$ such that $NS(X)\subseteq C_i$.
\end{theorem}

We will derive this along the way to the following:
\begin{theorem}
  \label{main}
For every semigroup $U$, there is a semigroup $V$ extending $U$ such that either $A\gin V$ or $A^c\gin V$.
\end{theorem}
Note that in general, if we can prove a theorem in second order arithmetic plus the existence of an idempotent ultrafilter, we might hope to replace the derivation with the theorem above: if we can arrange all ``queries'' to the idempotent ultrafilter along a countable well-ordering (where later queries can depend on the results of previous ones), we can then iterate the preceding theorem along the ordering, and the final semigroup will then provide answers to all queries in the proof.  (Such methods have been used for proofs using nonprincipal ultrafilters with no additional properties; see, for instance, \cite{avigad98}.)

In particular, Hindman's Theorem follows easily: find a semigroup $V$ containing either $C_1$ or $C_1^c$; in the latter case, extend it to contain either $C_2$ or $C_2^c$.  Proceed until $V$ contains $C_i$ for some $i$.  There is a set $X$ such that $n\in X$ implies $C_i-n\gin V$, and choose $x_0\in C_i\cap X$, find $X_0$ such that $n\in X_0$ implies $C_i-x_0\gin V$, then choose $x_1\in C_i\cap C_i-x_0\cap X\cap X_0$, and so on.  The result is an infinite set with all non-zero finite sums contained in $C_i$.

Hirst gives a similar equivalence for Iterated Hindman's Theorem:
\begin{theorem}[\cite{hirst04}]
  The following are equivalent:
  \begin{itemize}
  \item If $\{G_i\mid i\in\mathbb{N}\}$ is a collection of subsets of $\mathbb{N}$ then there is an increasing sequence $\langle x_i\rangle_{i\in\mathbb{N}}\subseteq\mathbb{N}$ such that for every $i$ either $FS(\langle x_j\rangle_{j\geq i})\subseteq G_i$ or $FS(\langle x_j\rangle_{j\geq i})\subseteq G_i^c$.
  \item If $\{A_i\mid i\in\mathbb{N}\}$ is a collection of subsets of $\mathbb{N}$ then there is a semigroup $U$ such that $A_i\gin U$ or $A_i^c\gin U$ for each $i$.
  \end{itemize}
\end{theorem}
In Section \ref{iht} we prove a slightly stronger form of the second case.

\section{Hindman's Theorem}
The proof that an idempotent ultrafilter exists is due (in a slightly different context) to Ellis \cite{ellis69}.  This argument consists of two steps: first, given a closed semigroup $U$, we find a closed sub-semigroup $V$ and a $p\in V$ such that there is a $q\in V$ with $p=q+p$.  The second step iterates the first step to find the idempotent $p$.

The first step uses the fact that for any closed semigroup, $U+p=\{q\mid\exists r\in U q=r+p\}$ is itself a closed semigroup, and if $p\in U$ then $U+p\subseteq U$.  If $p\not\in U+p$ then $U+p\subsetneq U$; so, given $U_0:=U$, we can take an arbitrary $p_0\in U_0$, if $p_0\not\in U_0+p_0$, take an arbitrary $p_1\in U_1:=U_0+p_0$, and so on.  By Zorn's Lemma and the compactness of the space of ultrafilters, this process must terminate in a non-empty closed semigroup: that is, eventually $p_\alpha\in U_\alpha+p_\alpha$.

The second step is analogous: if $U$ is a closed semigroup and $p\in U$, $\{q\in U\mid q+p=p\}$ is a closed sub-semigroup, and either contains $p$---in which case $p$ is idempotent---or fails to contain $p$, in which case it is a proper closed sub-semigroup, so again, by Zorn's Lemma (applying the first step again at every stage), we must eventually find a $U$ and a $p$ such that $p+p=p$.

To make this argument within second order arithmetic, however, we must eliminate the use of Zorn's Lemma.  We do this by replacing each occurrence with an induction; roughly speaking, we argue that if we pick any set $A$, either we can carry out the entire step (either the first or second) using ultrafilters $p$ such that $A\in p$, or eventually $U_\alpha$ is contained entirely in $\overline{A^c}$ (the set of ultrafilters containing $A^c$), in which case $A^c\in U_\alpha$.

(It is worth noting that it most presentations of the proof, the two steps above are folded into a single step with a single application of Zorn's Lemma, by simply requiring in advance the we work with a minimal closed semigroup; when attempting to construct the semigroup explicitly, however, a failure of the first step requires a different construction than a failure of the second step, and therefore it is helpful to divide the argument.)

First we need some minor lemmata showing that if no ultrafilter extending $U$ contains $A$ then in fact $U$ can be extended to contain $A^c$.
\begin{lemma}[\RCA]
Suppose that $U$ is a semigroup but $U\cup\{A\}$ fails to satisfy fip.  Then there is an $X\gin U$ such that $U\cup\{A^c-n\mid n\in X\cup\{0\}\}$ is a semigroup.
\label{extension1}
\end{lemma}
\begin{proof}
  Since $U\cup\{A\}$ fails to satisfy fip, but $U$ does satisfy fip, there must be some finite $F$ such that $U_F\cap A$ is finite.  Let $X:=\{n\mid U_F-n\gin U\}$; since $U$ is a semigroup, $X\gin U$.  Clearly $U\cup\{A^c-n\mid n\in X\}$ satisfies the semigroup property, so it suffices to check that it satisfies fip.

  Let $G$ and $Z\subseteq X$ be finite; we must check that $U_G\cap\bigcap_{n\in Z}A^c-n$ is infinite.  Since for each $n\in Z$, $U_F-n\gin U$, it follows that $U_G\cap\bigcap_{n\in Z}U_F-n$ is infinite.  But since $U_F-n\cap A-n$ is finite, $U_F-n\setminus A^c-n$ is also finite, so $U_G\cap\bigcap_{n\in Z}A^c-n$ must also be infinite.
\end{proof}

\begin{lemma}[\ACA]
Suppose that $U$ is a semigroup, $Y\gin U$, but $U\cup\{A,A-n\}$ fails to satisfy fip for every $n\in Y$.  Then there is an $X\gin U$ such that $U\cup\{A^c-n\mid n\in X\}$ is a semigroup.
\label{extension2}
\end{lemma}
\begin{proof}
  If $U\cup\{A\}$ fails to satisfy fip, the claim follows from Lemma \ref{extension1}.  So suppose $U\cup\{A\}$ satisfies fip.  Let $X$ be the set of $n$ such that $Y-n\gin U$.  Certainly $X\gin U$ since $U$ is a semigroup.  Clearly $U\cup\{A^c-n\mid n\in X\}$ satisfies the semigroup property, so it suffices to check that it satisfies fip.

  Let $G$ and $Z\subseteq X$ be finite; we must check that $U_G\cap\bigcap_{n\in Z}A^c-n$ is infinite.  For each $n\in Z$, there is an $F_n$ such that $U_{F_n}\cap A\cap A-n$ is finite, and therefore $U_{F_n}\cap A\setminus A-n$ is finite  Since $U_G\cap A\cap \bigcap_{n\in Z}U_{F_n}$ is infinite, it must be that $U_G\cap A\cap\bigcap_{n\in Z}A^c-n$ is infinite, so certainly $U_G\cap\bigcap_{n\in Z}A^c-n$ is infinite.
\end{proof}

\begin{lemma}[\RCA]
  If $U$ is a semigroup, $A$ is a set, and there is an $X\in U$ such that $A-n\gin U$ for each $n\in X$, then $U\cup\{A\}$ is a semigroup.
\label{extension3}
\end{lemma}
\begin{proof}
  If suffices to check that $U\cup\{A\}$ satisfies fip.  If $G$ is finite, choose $Y\gin U$ such that for each $n\in Y$, $U_G-n\gin U$.  Also choose $X'\gin U$ such that for each $n\in X'$, $X-n\gin U$.  Now choose $n\in X\cap Y$.  Then $U_G-n\cap A-n\gin U$, and is therefore infinite, so $U_G\cap A$ is infinite as well.
\end{proof}

The following theorem corresponds to the first application of Zorn's Lemma.
\begin{theorem}[\ATR]
  Let $U$ be a semigroup and $A$ a set.  Either there is a semigroup $V$ extending $U$ such that $V\cup\{A\}$ fails to satisfy fip, or there is an infinite set $S$ such that both
\[U\cup\{A-n\mid n\in FS(S)\}\]
and
\[U\cup\{FS(S)-n\mid n\in FS(S)\}\]
satisfy fip.
\label{part1}
\end{theorem}
\begin{proof}
  Fix an enumeration $F_1,\ldots,F_n,\ldots$ of the finite sets of integers.  Consider the tree of sequences $\sigma=\langle s_1<\ldots<s_n\rangle$ such that:
  \begin{itemize}
  \item $s_i\in U_{F_i}$ for every $i\leq n$, and
  \item $U\cup\{A-n\mid n\in FS(\sigma)\}$ satisfies fip.
  \end{itemize}
We will proceed by recursion along the Kleene-Brouwer ordering, $\prec$, of the well-founded part of the tree.  We construct an increasing sequence of semigroups $\{V_\sigma\}$ extending $U$ so that for every $\tau\preceq\sigma$,
\[V_\sigma\cup\{A-n\mid n\in FS(\tau)\}\]
fails to satisfy fip.

Let $\sigma$ belong to the well-founded part of this tree; for each $\tau\prec\sigma$ there is a $V_\tau$ satisfying the claim for $\tau$, and these $V_\tau$ form a chain increasing along $\prec$.  Since the union of a chain of semigroups is a semigroup, we have $V_{\prec\sigma}$ so that the claim holds for all $\tau\prec\sigma$.  Since every proper extension of $\sigma$ is $\prec\sigma$, there is no $n>\max\sigma$, $n\in U_{F_{lh(\sigma)+1}}$ such that 
\[V_{\prec\sigma}\cup\{A-m\mid m\in FS(\sigma\cup\{n\})\}\]
satisfies fip.  Equivalently, there is no $n$ such that
\[V_{\prec\sigma}\cup\{\bigcap_{m\in FS(\sigma)}A-m,\bigcap_{m\in FS(\sigma)}A-m-n\}\]
satisfies fip.  Then by Lemmata \ref{extension2} and \ref{extension3}, there is an extension $V_\sigma$ of $V_{\prec\sigma}$ with the desired property.  This completes the recursion.

If $\emptyset$ belongs to the well-founded part, we have $V_\emptyset$ such that $V_{\emptyset}\cup\{A\}$ fails to satisfy fip.  Otherwise, there is an infinite path $S$ through this tree.  It is easy to see that $S$ witnesses the claim: clearly $U\cup\{A-n\mid n\in FS(S)\}$ satisfies fip, since this is true for any initial segment of $S$.  For any $F$ and any $Z\subseteq FS(S)$, there are infinitely many $n>\max Z$ such that $F\subseteq F_n$, and for any such $n$, there is an $s_n\in S\cap U_F$ such that also $s_n+i\in FS(S)$ for each $i\in Z$, and therefore $s_n\in U_F\cap\bigcap_{i\in Z}FS(S)-i$.
\end{proof}

The following theorem corresponds to the second application of Zorn's Lemma.
\begin{theorem}[\Sigmati]
  Let $U$ be a semigroup and $A$ a set.  Either there is a semigroup $V$ extending $U$ such that $V\cup\{A\}$ fails to satisfy fip, or there is an infinite set $S$ such that $FS(S)\subseteq A$ and
\[U\cup\{FS(S)-n\mid n\in FS(S)\}\]
satisfies fip.
\label{part2}
\end{theorem}
\begin{proof}
  Note that $U\cup\{FS(S)-n\mid n\in FS(S)\}$ satisfying fip implies that $U\cup\{A-n\mid n\in FS(S)\}$ does as well, so this is a strengthening of Theorem \ref{part1}.  The proof is quite similar.

  Fix an enumeration $F_1,\ldots,F_n,\ldots$ of the finite sets of integers.  Consider the tree of sequences $\sigma=\langle s_1<\ldots<s_n\rangle$ such that:
  \begin{itemize}
  \item $NS(\sigma)\subseteq A$,
  \item $s_i\in U_{F_i}$ for every $i\leq n$, and
  \item $U\cup\{A-n\mid n\in FS(\sigma)\}$ satisfies fip.
  \end{itemize}
Again, we proceed by recursion along the Kleene-Brouwer ordering $\prec$ of the well-founded part of this tree, constructing an increasing chain of semigroups $\{V_\sigma\}$ refining $U$.  For each $\sigma$ in the well-founded part of this tree, and every $\tau\preceq\sigma$,
\[V_\sigma\cup\{A-n\mid n\in FS(\tau)\}\]
fails to satisfy fip.  (More precisely, to stay within \Sigmati, we prove the existence of the set $V_\sigma$ by induction; since the set constructed by Theorem \ref{part1} is not arithmetic in $U$, arithmetic recursion no longer suffices.)

Let $\sigma$ belong to the well-founded part of this tree; for each $\tau\prec\sigma$ there is a $V_\tau$ satisfying the claim for $\tau$, and these $V_\tau$ form a chain increasing along $\prec$.  Since the union of a chain of semigroups is a semigroup, we have $V_{\prec\sigma}$ so that the claim holds for all $\tau\prec\sigma$.  By Theorem \ref{part1}, either there is a $V_\sigma$ extending $V_{\prec\sigma}$ such that $V_\sigma\cup\{\bigcap_{m\in FS(\sigma)}A-m\}$ fails to satisfy fip, in which case we are done, or there is an $S$ such that $V_{\prec\sigma}\cup\{FS(S)-n\mid n\in FS(S)\}$ and $V_{\prec\sigma}\cup\{\bigcap_{m\in FS(\sigma)}A-m-n\mid n\in FS(S)\}$ satisfy fip.

In the latter case, let $V_\sigma:=V_{\prec\sigma}\cup\{FS(S)-n\mid n\in FS(S)\}$.  This is clearly a semigroup.  If $V_\sigma\cup\{\bigcap_{m\in FS(\sigma)}A-m\}$ satisfied fip, we could find an $n\in U_{F_{lh(\sigma)+1}}\cap FS(S)\cap \bigcap_{m\in FS(\sigma)}A-m$, $n>\max\sigma$, and it would follow that $V_{\prec\sigma}\cup\{A-m\mid m\in FS(\sigma\cup\{n\})\}$ satisfied fip.  But $\sigma^\frown\langle n\rangle\prec\sigma$, so by IH, this cannot satisfy fip.  Therefore $V_\sigma\cup\{\bigcap_{m\in FS(\sigma)}A-m\}$ fails to satisfy fip as well.  This completes the induction.

If $\emptyset$ belongs to the well-founded part, $V_\emptyset$ satisfies the claim.  Otherwise there is an infinite path $S$ through the tree which satisfies the claim.
\end{proof}

\begin{theorem}
  If $U$ is a semigroup and $A$ is a set, there is a semigroup $V$ extending $U$ such that either $A\in V$ or $A^c\in V$.
\end{theorem}

We note the relationship between this proof and Baumgartner's: in place of the inductions above, we could seek a property of $A$ which would guarantee that we are in the case where the set $S$ exists (and $V$ does not).  We could say that $A$ was large for $U$ if no semigroup extending $U$ contained $A^c$; it is not hard to see that if $A$ is not large for $U$, there is an extension $V$ such that $A^c$ is large for $V$.  Then, under the assumption that $A$ is large, the inductions above would ``unwrap'' into, essentially, Baumgartner's argument.

\section{Iterated Hindman's Theorem}
\label{iht}

With a slight modification, it is possible to obtain the Iterated Hindman's Theorem.
\begin{definition}
  If $b\in\{1,-1\}$, define
\[b\cdot A:=\left\{\begin{array}{ll}
A&\text{if }b=1\\
A^c&\text{if }b=-1
\end{array}\right.\]
\end{definition}

\begin{theorem}[\Sigmati]
  Let $U$ be a semigroup and $\{A_i\}$ a sequence of sets.  There are infinite sets $S=\{s_i\}, B=\{b_i\}$ such that $FS(\{s_j\}_{j\geq i})\subseteq b_i\cdot A_i$ for each $i$, and
\[U\cup\{FS(S)-n\mid n\in FS(S)\}\]
satisfies fip.
\end{theorem}
\begin{proof}
  Fix an enumeration $F_1,\ldots,F_n,\ldots$ of the finite sets of integers.  Consider the tree of sequences $\sigma=\langle b_1,s_1,\ldots,b_k,s_k\rangle$ or $\sigma=\langle b_1,s_1,\ldots,b_k\rangle$ such that:
  \begin{itemize}
  \item $s_i<s_{i+1}$ whenever $s_{i+1}$ is defined,
  \item $NS(\{s_j\}_{j\geq i})\subseteq b_i\cdot A_i$ for each $i\leq k$,
  \item $s_i\in U_{F_i}$ for every $i\leq k$, and
  \item $U\cup\{b_i\cdot A_i-m\mid i\leq k\wedge m\in FS(\{s_j\}_{j\geq i})\}$ satisfies fip.
  \end{itemize}

By induction along the Kleene-Brouwer ordering $\prec$ on the well-founded part of this tree, we show that for each $\sigma$ in the well-founded part, there is a semigroup $V_{\sigma}$ extending $U$ such that for every $\tau\preceq\sigma$, $\tau=\{c_1,t_1,\ldots\}$,
\[V_\sigma\cup\{b_i\cdot (A_i-m)\mid m\in FS(\{t_j\}_{j\geq i})\}\]
fails to satisfy fip.

Let $\sigma$ belong to the well-founded part.  If $\sigma=\langle b_1,s_1,\ldots,b_k,s_k\rangle$ then both $\sigma^\frown\langle -1\rangle$ and $\sigma^\frown\langle 1\rangle$ belong to the well-founded part of the tree, and so $V_{\prec\sigma}\cup\{b_i\cdot (A_i-m)\mid i\leq k\wedge m\in FS(\{s_j\}_{j\geq i})\}\cup \{A_{k+1}\}$ and $V_{\prec\sigma}\cup\{b_i\cdot (A_i-m)\mid i\leq k \wedge m\in FS(\{s_j\}_{j\geq i})\}\cup\{A_{k+1}^c\}$ fail to satisfy fip; but this implies that $V_{\prec\sigma}\cup\{b_i\cdot (A_i-m)\mid i\leq k\wedge m\in FS(\{s_j\}_{j\geq i})\}$ fails to satisfy fip, so we may take $V_\sigma:=V_{\prec\sigma}$.

So suppose $\sigma$ belongs to the well-founded part and $\sigma=\langle b_1,s_1,\ldots,b_k\rangle$.  Set $A:=\bigcap_{i\leq k, m\in FS(\{s_j\}_{j\geq i})}b_i\cdot (A_i-m)$.  By Theorem \ref{part1}, we may find either an extension $V$ of $V_{\prec\sigma}$ such that $V\cup\{A\}$ fails to satisfy fip, or an $S$ such that both $V_{\prec\sigma}\cup\{A-n\mid n\in FS(S)\}$ and $V_{\prec\sigma}\cup\{FS(S)-n\mid n\in FS(S)\}$ satisfy fip.  In the former case, we are done; in the latter, let $V_\sigma:=V_{\prec\sigma}\cup\{FS(S)-n\mid n\in FS(S)\}$.  If it were the case that $V_\sigma\cup\{A\}$ satisfied fip, we could take $m\in A\cap U_{F_k}\cap FS(S)$, $m>s_{k-1}$, and observe that $V_{\prec\sigma}\cup\{A,A-m\}$ must satisfy fip.  Since this cannot be the case, it must be that $V_\sigma\cup\{A\}$ fails to satisfy fip, as promised.

Then $\langle\rangle$ cannot be in the well-founded part of this tree, since there can be no semigroup $V$ with the property that $V$ fails to satisfy fip.  So there is an infinite path $\{b_1,s_1,\ldots\}$ through the tree; it is immediate to see that $\{s_i\},\{b_i\}$ are the desired witnesses.
\end{proof}

\begin{theorem}
  If $U$ is a semigroup and $\{A_i\mid i\in\mathbb{N}\}$ is a collection of sets, there is a semigroup $V$ extending $U$ such that for each $i$, either $A_i\gin V$ or $A_i^c\gin V$.
\end{theorem}

\bibliographystyle{plain}
\bibliography{combinatorics,ergodic,revmath}

\end{document}